\documentclass[12pt]{amsart}
\usepackage{hyperref}
\numberwithin{equation}{section}
\newtheorem{theorem}{Theorem}
\numberwithin{theorem}{section}
\newtheorem{lemma}{Lemma}
\numberwithin{theorem}{section} \numberwithin{lemma}{section}

\numberwithin{definition}{section}
\newtheorem{corollary}{Corollary}
\numberwithin{corollary}{section}
\newtheorem{remark}{Remark}
\numberwithin{remark}{section}

\numberwithin{proposition}{section}
\def\b{\begin{equation}}
\def\e{\end{equation}}
\newcommand{\ignore}[1]{}
\normalsize
\date {August 10, 2021}
\begin{document}
\pagenumbering{arabic} \pagenumbering{arabic}\setcounter{page}{1}
\tracingpages 1
\title{Several weighted Hardy, Hardy-Sobolev-Maz'ya and related inequalities}
\author{Ismail K\"ombe, S\"umeyye Bakim and Reyhan Tellio\u{g}lu Baleko\u{g}lu}
\dedicatory {}
\address{Ismail K\"ombe, Department of Mathematics\\ Faculty of Humanities  and Social Sciences\\ Istanbul Commerce
	University\\
	Beyoglu, Istanbul,
	Turkey}
\email{ikombe@ticaret.edu.tr}
\address{Sümeyye Bakim, Engineering Faculty\\KTO Karatay University\\ Konya, TURKEY } \email{sumeyye.bakim@karatay.edu.tr}
\address{Reyhan Tellioglu Balekoglu, Department of Mathematics\\ Faculty of Humanities  and Social Sciences\\ Istanbul Commerce
	University\\
	Beyoglu, Istanbul,
	Turkey   } \email{reyhan.tellioglu@istanbulticaret.edu.tr}

\begin{abstract}
In this paper we establish several Hardy and Hardy-Sobolev type inequalities with homogeneous  weights on the first orthant
$\displaystyle \mathbb{R}_{*}^n:=\{(x_1, \ldots, x_n):x_1>0, \ldots, x_n>0 \}$. We then use some of them to produce  Hardy type inequalities with  remainder terms. 
Furthermore, we obtain  some interpolation inequalities  and  Maz'ya type inequalities with remainder terms with the help of Maz'ya inequality and Sobolev inequality of Cabr\'e and Ros-Orton on the upper half space $\mathbb{R}_{+}^n:=\{x=(x_1, \ldots, x_n)|\, x_n>0 \}$.
\end{abstract}
\maketitle
\section{Introduction}
The classical  Hardy inequality states that
\begin{equation}\label{1.1}\int_{\mathbb{R}^n}
|\nabla\phi|^p dx \ge \left| \frac{n-p}{p}\right|^p \int_{\mathbb{R}^n}
\frac{{|\phi|}^p}{|x|^p} dx, \end{equation}
and holds for all $\phi\in C_0^{\infty}(\mathbb{R}^n)$ if $1<p<n$, and for all $\phi \in   C_0^{\infty}(\mathbb{R}^n\setminus \{0\})$ if $p>n$. The constant $|\frac{n-p}{p}|^p$ in the above inequality is sharp and is never achieved by nontrivial functions.
In the critical case $p=n$, the inequality $(\ref{1.1})$ fails for any constant. However, there is another version of the Hardy inequality slightly different from $(\ref{1.1})$,
\begin{equation}\label{1.2}
\int_{B_1} |\nabla\phi|^p dx\ge
\left(\frac{p-1}{p}\right)^p\int_{B_1}\frac{|\phi|^p}{|x|^p\left(\ln(\frac{1}{|x|})\right)^p}dx,
\end{equation} which is essentially proved by Leray \cite{Leray} in $\mathbb{R}^2$ and later by Adimurthi and Sandeep \cite{Adimurthi-Sandeep} in $\mathbb{R}^n$, $n\ge 3$. Here $B_1$ is the unit ball in $\mathbb{R}^n$  centered at the origin and $\phi\in C_0^{\infty}(B_1\setminus\{0\})$, and the constant $(\frac{p-1}{p})^p$ is sharp.

Another inequality involving non-radial weight is the  Maz'ya  inequality \cite{Mazya},

\begin{equation}\label{1.3}
\int_{\mathbb{R}^n} |x_n|^{p-1} |\nabla \phi|^p dx \geq \frac{1}{(2p)^p}\int_{\mathbb{R}^n} \frac{|\phi|^p}{(x_{n-1}^2+x_n^2)^{\frac{1}{2}}} dx,
\end{equation}
where   $\phi \in C_0^{\infty}(\mathbb{R}^n)$ and $1<p<\infty$.

These inequalities play important roles in many areas such as analysis,  geometry and  mathematical physics. By virtue of aforementioned reasons, there is a vast amount of literature on these  inequalities together with their generalizations, extensions, and refinements.  We refer the interested reader to the monographs \cite{Kufner}, \cite{Davies}, \cite{Moradifam}, \cite{Lewis}, \cite{Ruzhansky}, the papers \cite{FMT}, \cite{Rupert}, \cite{Kombe-Ozaydin}, \cite{Goldstein-Kombe-Yener}, \cite{Glu}, \cite{GGKB-2}, and the references cited therein.

Before we state our results, let us first briefly recall the related results in the literature. For instance, Davies \cite{Davies} proved the following extension of Hardy's inequality  for convex
domains  $ \Omega \subset \mathbb{R}^n$, $n\ge 2$,
\begin{equation}\label{1.4}
\int_{\Omega} |\nabla \phi|^2 dx\ge
\frac{1}{4}\int_{\Omega}\frac{|\phi|^2}{\delta^2}dx,
\end{equation}
where $\phi \in C_{0}^{\infty }\left( \Omega \right)$.  Here
$\delta\left(x\right) = \mathrm{dist}\, \left(x,\partial\Omega\right) = \min_{y\in\partial\Omega}
\left\vert x-y\right\vert$
is the distance function to the boundary $\partial \Omega$. Moreover the constant $\frac{1}{4}$ is sharp (see \cite{Matskewich}).  When $\Omega$ is the upper half-space $\mathbb{R}_{+}^n$, the inequality \eqref{1.4} takes the form
 \begin{equation}\label{1.5}
\int_{\mathbb{R}_{+}^n} |\nabla \phi|^2 dx\ge
\frac{1}{4}\int_{\mathbb{R}_{+}^n}\frac{|\phi|^2}{x_n^2}dx.
\end{equation}

An improved version of the inequality \eqref{1.5}, which has been obtained by Maz'ya much earlier in his book \cite{Mazya},

\begin{equation}\label{1.6}
\int_{\mathbb{R}_{+}^n} |\nabla \phi|^2 dx \geq \frac{1}{4}\int_{\mathbb{R}_{+}^n} \frac{|\phi|^2}{x_{n}^2} dx+ C_n\left(\int_{\mathbb{R}_{+}^n} x_n^{\gamma}|\phi|^{p}  dx\right)^{\frac{2}{p}},
\end{equation}
where $\phi \in C_{0}^{\infty }( \mathbb{R}_{+}^n)$, $n\ge 3$, $2<p\le \frac{2n}{n-2}$, $\gamma=\frac{(n-2)p}{2}-n$  and $C_n$ is a positive constant. If $p=\frac{2n}{n-2}$ then  we have the following celebrated  Hardy-Sobolev-Maz'ya inequality

\begin{equation}\label{1.7}
\int_{\mathbb{R}_{+}^n} |\nabla \phi|^2 dx \geq \frac{1}{4}\int_{\mathbb{R}_{+}^n} \frac{|\phi|^2}{x_{n}^2} dx+ C_n\left(\int_{\mathbb{R}_{+}^n} |\phi|^{\frac{2n}{n-2}}  dx\right)^{\frac{n-2}{n}}.
\end{equation}

On the other hand,  substituting  $\phi(x)=x_n^{-1/2} \varphi(x)$ into (\ref{1.3}) gives another improved version of \eqref{1.5}
\begin{equation}\label{1.8}
\int_{\mathbb{R}_{+}^n} |\nabla \varphi|^2 dx \geq \frac{1}{4}\int_{\mathbb{R}_{+}^n} \frac{|\varphi|^2}{x_{n}^2} dx+ \frac{1}{16}\int_{\mathbb{R}_{+}^n} \frac{|\varphi|^2}{(x_{n-1}^2+x_n^2)^{\frac{1}{2}} x_n} dx,
\end{equation}
where $\varphi \in C_0^{\infty}(\mathbb{R}_{+}^n)$.

Regarding  the $L^2$ improved versions of \eqref{1.1} for a bounded domain $\Omega$, we have the following  sharp Hardy-Poincar\'{e} and Hardy-Sobolev inequalities, which  have been proved by Brezis and V\'azquez \cite{BV}
\begin{equation}\label{1.9}
\int _{\Omega} |\nabla \phi|^2dx\ge
\left(\frac{n-2}{2}\right)^2\int_{\Omega}
\frac{|\phi|^2}{|x|^2}dx+\mu\left(\frac{\omega_n}{|\Omega|}\right)^{2/n}\int
_{\Omega} |\phi|^2dx,
\end{equation}
and
\begin{equation}\label{1.10}
\int _{\Omega} |\nabla \phi|^2dx\ge
\left(\frac{n-2}{2}\right)^2\int_{\Omega}
\frac{|\phi|^2}{|x|^2}dx+C\left(\int _{\Omega}
|\phi|^qdx\right)^{\frac{2}{q}},
\end{equation}
where $\phi\in C_0^{1} (\Omega)$, $C=C(\Omega, n)>0$,  $\omega_n$
and $|\Omega|$ denote the $n$-dimensional Lebesgue measure of the
unit ball $B\subset \mathbb{R}^n$  and the domain $\Omega$,
respectively. Here $\mu$ is the first eigenvalue of the Laplace
operator in the two dimensional unit disk, and it is optimal when
$\Omega$ is a ball centered at the origin,
$2\le q< \frac{2n}{n-2}$ and $q$ cannot be replaced by
$\frac{2n}{n-2}$.

In a series of papers \cite{Tidblom2}-\cite{Tidblom}  Tidblom studied various Hardy type inequalities with remainder terms and proved the following inequality in \cite{Tidblom2},
\begin{equation}\label{1.12}
\int_{\mathbb{R}_{*}^{n}} |\nabla \phi |^{2}dx \ge \frac{1}{4} \int _{\mathbb{R}_{*}^{n}} \frac{|\phi|^2}{|x|^2}+\left(\frac{2n-1}{4n^2}\right)  \int_{\mathbb{R}_{*}^{n}}  \left(\frac{1}{x_1^2}+\frac{1}{x_2^2}+...+\frac{1}{x_n^2}\right )|\phi|^{2}dx,
\end{equation}
where  $\phi \in C_0^{\infty} (\mathbb{R}_{*}^{n})$. Moreover, Tidblom asked whether the inequality
\begin{equation}\label{1.13}
\int_{\Omega} |\nabla \phi|^p dx \geq C  \int_{\Omega} \left(\frac{1}{x_1^p}+\frac{1}{x_2^p}\right) |\phi|^p dx,
\end{equation}
holds for some positive constant $C$. In \cite{Tidblom}, Tidblom  obtained the following inequality, which is the $L^p$ analogue of  Maz'ya's inequality \eqref{1.8},
\begin{equation}\label{1.11}
\int_{\mathbb{R}_{+}^n} |\nabla \phi|^p dx \geq \left(\frac{p-1}{p}\right)^p  \int_{\mathbb{R}_{+}^n} \frac{|\phi|^p}{x_n^p}dx + \alpha(p,\tau)\int_{\mathbb{R}_{+}^n} \frac{|\phi|^p}{x_n^{p-\tau}(x_{n-1}^2+x_n^2)^{\frac{\tau}{2}}} dx,
\end{equation}
where $p>1$, $\tau > 0$, $\alpha (p,\tau)$ is a positive constant and $\phi \in C_0^{\infty} (\mathbb{R}_{+}^n)$. 

On the other hand, Filippas, Tertikas and Tidblom \cite{Tertikas-Filippas-Tidblom} proved the following  inequality which has a critical Sobolev exponent,
\begin{equation}\label{1.14}
\int_{\mathbb{R}_{k_+}^n} |\nabla \phi|^2 dx\geq\frac{1}{4} \int_{\mathbb{R}_{k_+}^n} \left(\frac{1}{x_1^2}+\dots+\frac{1}{x_k^2}\right) |\phi|^2 dx + C\left( \int_{\mathbb{R}_{k_+}^n} |\phi|^{\frac{2n}{n-2}} dx\right)^{\frac{n-2}{n}},
\end{equation}
where $\mathbb{R}_{k_+}^n=\{x=(x_1,\ldots, x_n)|\, x_1>0, \ldots, x_k>0\} \subset \mathbb{R}^n$, $\phi\in C_0^{\infty}(\mathbb{R}_{k_+}^n)$ and $C$ is a positive constant.

Recently, Cabr\'e and Ros-Orton \cite{CabreRosOrton} obtained a sharp form  of the Sobolev inequality with monomial weight $x^A$, where $A=(A_1, \dots, A_n)$ and  \[x^A:=|x_1|^{A_1}\cdots |x_n|^{A_n}, \quad A_1\ge 0,\dots, A_n\ge 0.\]
More precisely, the following inequality  holds for all $\phi\in C_0^1(\mathbb{R}_*^n)$
\begin{equation}\label{1.15}
\left(\int_{\mathbb{R}_*^n}x^A|\nabla \phi|^pdx\right)^{\frac{1}{p}} \ge S_p(D) \left(\int_{\mathbb{R}_*^n}x^A|\phi|^{p^{*}}dx\right)^{\frac{1}{p^{*}}},
\end{equation}
where  $\mathbb{R}_*^n=\{x=(x_1,\dots,x_n) \in \mathbb{R}^n : x_i > 0 \quad \text{when} \quad A_i \ge 0\}$,  $D=n+A_1+\dots+A_n$, $1<p<D$,  $p^{*}=\frac{Dp}{D-p}$, the  constant $S_p(D)>0$ is sharp (its explicit value can be found in \cite{CabreRosOrton}) and the equality
 holds if and only if $\phi_{a,b}(x)=\left(a+b|x|^{\frac{p}{p-1}}\right)^{1-\frac{D}{p}}$,  where $a$ and $b$ are any positive constants.

Motivated by the above results, we have three main purposes in this paper. The first one  is  to study weighted Hardy type inequalities  on the first orthant \[\mathbb{R}_{*}^n:=\{x=(x_1, \ldots, x_n)|\, x_1>0, \ldots, x_n>0 \}.\]  The second purpose is to obtain various explicit remainder terms for Hardy type inequalities on some subdomains of $\mathbb{R}_{*}^n$. The last purpose is to study Maz'ya-Sobolev inequality and Maz'ya  type inequalities with remainder terms in the upper half-space \[\mathbb{R}_{+}^n:=\{x=(x_1, \ldots, x_n)|\, x_n>0 \}.\]

\section{Weighted Hardy and Poincar\'e Type Inequalities}\label{Section2}
Throughout this section and the next section, we assume that  $\Omega$ is a  bounded domain with smooth boundary $\partial \Omega$ in  $\mathbb{R}_{*}^n:=\{x=(x_1, \ldots, x_n)|\, x_1>0, \ldots, x_n>0 \}$, $n\ge 2$. The generic point is $x=\left( x_{1},\ldots ,x_{n}\right) \in
\mathbb{R}^{n}$ and $r=\left\vert x\right\vert =\sqrt{x_{1}^{2}+\cdots +x_{n}^{2}}.$  We  denote   $B_{R}=\left\{ x\in
\mathbb{R}
^{n}:\left\vert x\right\vert <R\right\} $ is the open ball at the origin with radius $R$ in $\mathbb{R}^n$.

We begin this section by proving a Hardy type inequality with the homogeneous weight $w(x_1,\dots,x_n)=x_1^{\alpha}\cdots x_n^{\alpha}$ where  $\alpha \in \mathbb{R}$. In the special case where  $\alpha=0$, our result gives  an  affirmative answer to Tidblom's question (\ref{1.13}).
\smallskip
\begin{theorem}\label{Theorem 2.1}
Let $\phi \in C_0^{\infty}(\Omega)$, $p>2$, $\alpha \in \mathbb{R}$ and $p-\alpha-1>0$.   Then we have,
\begin{equation}\label{2.1}
\int_{\Omega}(x_1\dots x_n)^{\alpha} |\nabla \phi |^{p} dx \ge \left(\frac{p-\alpha-1}{p}\right)^p \int_{\Omega} (x_1 \dots x_n)^{\alpha}\left(\frac{1}{x_1^p}+\dots+\frac{1}{x_n^p}\right )|\phi|^{p} dx.
\end{equation}

\end{theorem}

\begin{proof} Consider the  vector field   \[\textbf{F}=(x_1^{1-p},x_2^{1-p},\dots,x_n^{1-p}).\]
Direct computation shows that
\begin{equation}\label{2.2}
\text{div} \, \textbf{F}=(1-p)\left(\frac{1}{x_1^p}+\frac{1}{x_2^p}+\dots+\frac{1}{x_n^p}\right ).
\end{equation}
Multiplying  both sides of (\ref{2.2}) by  $(x_1\dots x_n)^{\alpha}|\phi|^p$, and integrating by parts and rearranging terms in a standard way we get
\begin{equation} \label {2.3} \frac{p-1-\alpha}{p}\int_{\Omega} \left(\frac{1}{x_1^p}+\dots+\frac{1}{x_n^p}\right ) (x_1\dots x_n)^{\alpha}|\phi|^p dx=\int_{\Omega} (x_1\dots x_n)^{\alpha}|\phi|^{p-1} \nabla \phi\cdot \textbf{F} dx. \end{equation}
Denote by $L$ and $R$ the left-hand side and the right-hand side of $\eqref{2.3}$, respectively. An application of H\"older's and Young's inequality yields,

\begin{equation}\label{2.4}  L=R \leq  \epsilon \int_{\Omega} (x_1\dots x_n)^{\alpha}|\nabla \phi|^p dx+ (\frac{p-1}{p})(p\epsilon)^{\frac{1}{1-p}}\int_{\Omega} (x_1\dots x_n)^{\alpha}|\textbf{F}|^{\frac{p}{p-1}} |\phi|^{p}dx\end{equation}
where $\epsilon>0$  and will be chosen later. On the other hand, one can show that

\begin{equation}\label{2.5}
|\textbf{F}|^{\frac{p}{p-1}} \le \frac{1}{x_1^p}+ \frac{1}{x_2^p}+\dots+\frac{1}{x_n^p}.
\end{equation} Substituting (\ref{2.5}) into (\ref{2.4}) and rearranging the terms gives

\[H_{p,\alpha}(\epsilon)\int_{\Omega}(x_1\dots x_n)^{\alpha}\left(\frac{1}{x_1^p}+\frac{1}{x_2^p}+\dots+\frac{1}{x_n^p}\right )|\phi|^p dx
\leq \int_{\Omega} (x_1\dots x_n)^{\alpha}|\nabla \phi|^p dx,\]
where $H_{p,\alpha}(\epsilon)=  \frac{p-1}{p}\left( \frac{1-(p\epsilon)^{\frac{1}{1-p}}}{\epsilon}\right)-\frac{\alpha}{p\epsilon}.$
Note that the function $H_{p,\alpha}(\epsilon)$  attains the maximum for $\epsilon=\frac{1}{p}\left( \frac{p-\alpha-1}{p}\right)^{1-p}$, and this maximum is equal to $(\frac{p-\alpha-1}{p})^{p} $. Now we have the desired
inequality (\ref{2.1}),
\[
\int_{\Omega}(x_1\dots x_n)^{\alpha} |\nabla \phi |^{p} dx \ge \left(\frac{p-\alpha-1}{p}\right)^p \int_{\Omega} (x_1 \dots x_n)^{\alpha}\left(\frac{1}{x_1^p}+\dots+\frac{1}{x_n^p}\right )|\phi|^{p} dx.
\]
\end{proof}
\smallskip

We now present the following weighted Hardy-Sobolev inequality that can be obtained by H\"older interpolation between the Hardy type inequality \eqref{2.1}  and the Sobolev inequality \eqref{1.15} with $A_1=\dots =A_N=\alpha$.
\begin{corollary}
	Let  $2<p<n(1+\alpha)$, $0\le \alpha <p-1$, $0\le s\le p$ and  $p^{*}(s)=p\left(\frac{D-s}{D-p}\right)$. Then the following inequality holds,
	\begin{equation}\label{2.6}
	\left(\int_{\Omega}(x_1\dots x_n)^{\alpha} |\nabla \phi |^{p} dx\right)^{\frac{1}{p}} \ge C\left(\int_{\Omega} (x_1 \dots x_n)^{\alpha}\left(\frac{1}{x_1^p}+\dots+\frac{1}{x_n^p}\right )^{\frac{s}{p}}|\phi|^{p^{*}(s)} dx\right)^{\frac{1}{p^{*}(s)}},
	\end{equation}
for all $\phi \in C_0^{\infty}(\Omega)$. Here $C=\left(\frac{p-\alpha-1}{p}\right)^{\frac{s(D-p)}{p(D-s)}}\left( S_p(D)\right )^{\frac{D(p-s)}{p(D-s)}}$ and  $D=n(1+\alpha)$.
	\end{corollary}
When $s =0$, then \eqref{2.6} is the  Sobolev inequality \eqref{1.15} with $A_1=\dots =A_N=\alpha$, whereas for $s = p$, we get the Hardy type inequality \eqref{2.1}.
\smallskip

Using the inequality (\ref{2.1})   and the arithmetic-geometric mean inequality, we obtain the following weighted  Hardy type inequality.

\begin{theorem}\label{Theorem 2.2}
	 Let $\phi \in C_0^{\infty}(\Omega)$, $p>2$ , $\alpha\in \mathbb{R}$ and  $p-\alpha-1>0$. Then  we have,

\begin{equation}\label{2.7}
\int_{\Omega}\left(x_1\dots x_n\right)^{\alpha}|\nabla\phi|^pdx \geq n \left(\frac{p-\alpha-1}{p}\right)^p  \int_{\Omega}\left(x_1\dots x_n\right)^{\alpha-\frac{p}{n}}|\phi|^pdx.
\end{equation}
\end{theorem}

\begin{proof}
	
The inequality between the arithmetic and geometric means of $n$ positive numbers $x_i^{-p}$, $i=1,2,\dots,n$ may be written as
\begin{equation}\label{2.8}
\frac{n}{(x_1x_2\dots x_n)^{p/n}} \le \frac{1}{x_1^{p}}+\frac{1}{x_2^{p}}+\dots+\frac{1}{x_n^{p}}.
\end{equation}	
Multiplying both sides of (\ref{2.8}) by $\left(x_1\dots x_n\right)^{\alpha}|\phi|^p$  and integrating, we obtain
	
\begin{equation}\label{2.9}
n\int_{\Omega} \frac{|\phi|^p}{(x_1x_2\dots x_n)^{\frac{p}{n}-\alpha}}dx \le \int_{\Omega}\left(x_1\dots x_n\right)^{\alpha}\left(\frac{1}{x_1^{p}}+\frac{1}{x_2^{p}}+\dots+\frac{1}{x_n^{p}}\right) |\phi|^p dx.
\end{equation}		
By applying the inequality (\ref{2.1}) in Theorem $\ref{Theorem 2.1}$  to the right hand side of (\ref{2.9}), it follows that
	
\begin{equation}\label{2.10}
\int_{\Omega}\left(x_1\dots x_n\right)^{\alpha}|\nabla\phi|^pdx \geq n \left(\frac{p-\alpha-1}{p}\right)^p  \int_{\Omega}\left(x_1\dots x_n\right)^{\alpha-\frac{p}{n}}|\phi|^pdx.
\end{equation}
The proof of Theorem $\ref{Theorem 2.2}$  is now complete.
\end{proof}	
\smallskip

We continue to find other weighted inequalities, and we will achieve this by selecting some special vector fields. Our results are as follows.

\begin{theorem}\label{Theorem 2.3}
Let $p>1$. Then for all $\phi \in C_0^{\infty}(\Omega)$, we have

\begin{equation}\label{2.11}	
\int_{\Omega} |x|^p(x_1x_2 \dots x_n) |\nabla \phi|^p dx
\ge \left(\frac{2n}{p}\right)^p \int_{\Omega}
(x_1x_2\dots x_n)|\phi|^pdx.	\end{equation}
Moreover, for every $\phi \in C_0^{\infty} (B_R \cap \mathbb{R}_*^n)$
\begin{equation}\label{2.12}\int_{B_R \cap \mathbb{R}_*^n} (x_1 x_2 \dots x_n)|\nabla \phi |^{p}dx \ge \left(\frac{2n}{pR}\right)^p \int_{B_R \cap \mathbb{R}_*^n} (x_1 x_2 \dots x_n)|\phi|^pdx.
\end{equation}
	
\end{theorem}

\begin{proof}
The method of proof is essentially the same as that used in proving Theorem \ref{Theorem 2.1}, but using the vector field
\[\textbf{F}=(x_1^{2}x_2\dots x_n,x_1x_2^{2}\dots x_n,\dots,x_1x_2\dots x_n^{2}).\]
We omit the details.

\end{proof}
\smallskip

Next, we have the following theorem.
\begin{theorem}\label{Theorem 2.4}
Let $p>1$. Then for all $\phi \in C_0^{\infty}(\Omega)$, we have
\begin{equation}\label{2.13}
\int_{\Omega} \frac{|x|^p}{x_1^px_2^p\dots x_n^p}|\nabla \phi |^pdx \ge \left(\frac{n(p-1)}{p}\right)^p \int_{\Omega} \frac{|\phi|^p}{x_1^px_2^p \dots x_n^p}dx.
\end{equation}
\end{theorem}

\begin{proof}
The proof follows using the same argument as in the proof of Theorem \ref{2.1} but using  the vector field 	
\[\textbf{F} =(-x_1^{1-p}x_2^{-p}\dots x_n^{-p},-x_1^{-p}x_2^{1-p}\dots x_n^{-p},\dots,-x_1^{-p}x_2^{-p}\dots x_n^{1-p}).\]
We omit the details.
\end{proof}
\smallskip

As a final result of this section, we are going to use Theorem $\ref{Theorem 2.1}$ to prove the following interpolation inequality.

\begin{theorem}\label{Theorem 2.5}
Let $p>2$. Then for all $\phi \in C_0^{\infty}(\Omega)$, we have
\begin{equation}\label{2.14}
\begin{aligned}\left(\int_{\Omega} |\nabla \phi |^{p} dx\right)\left(\int_{\Omega}\left(\frac{1}{x_1^p}+\dots+\frac{1}{x_n^p}\right)|\phi|^p dx\right)^{p-1} \ge K\left(\int_{\Omega}\frac{|\phi|^p}{x_1^p+\dots+x_n^p} dx\right)^p\end{aligned},
\end{equation}
where $K=(\frac{n^2(p-1)}{p})^p$.
\end{theorem}

\begin{proof}
Let $u=(x_1^{p/2}, x_2^{p/2},\dots,x_n^{p/2})$ and $v=(\frac{1}{x_1^{p/2}}, \frac{1}{x_2^{p/2}}, \dots, \frac{1}{x_n^{p/2}})$ be vectors in $\Omega$. Application of the Cauchy-Schwarz inequality yields
	
\begin{equation}\label{2.15} \frac{1}{x_1^p}+\frac{1}{x_2^p}+\dots + \frac{1}{x_n^p} \ge \frac{n^2}{x_1^p+x_2^p+\dots+x_n^p}. \end{equation}
Multiplying both sides of (\ref{2.15}) by $|\phi|^p$ and integrating over $\Omega$, we obtain
\begin{equation}\label{2.16}
\begin{aligned}
n^2\int_{\Omega}\frac{|\phi|^p}{(x_1^p+x_2^p+\dots+x_n^p)} dx \leq \int_{\Omega}\left(\frac{1}{x_1^p}+\dots+\frac{1}{x_n^p}\right)|\phi|^p dx.
\end{aligned}
\end{equation}
Let us denote the right hand side of  (\ref{2.16})  by $R$.
It is clear that

\begin{equation} \label{2.17}
R=\left(\int_{\Omega}\left(\frac{1}{x_1^p}+\dots+\frac{1}{x_n^p}\right)|\phi|^p dx\right)^{\frac{1}{p}}\left(\int_{\Omega}\left(\frac{1}{x_1^p}+\dots+\frac{1}{x_n^p}\right)|\phi|^p dx\right)^{1-\frac{1}{p}}.
\end{equation}
Applying the inequality (\ref{2.1}) with $\alpha=0$   to the first integral term in ($\ref{2.17}$) we find
	
\begin{equation} \label{2.18}
R\le \frac{p}{p-1} \left(\int_{\Omega}|\nabla\phi|^p dx\right)^{1/p}\left(\int_{\Omega}\left(\frac{1}{x_1^p}+\dots+\frac{1}{x_n^p}\right)|\phi|^p dx\right)^{\frac{p-1}{p}}.
\end{equation}
Combining (\ref{2.16}) and (\ref{2.18}) gives the desired inequality,
\[\begin{aligned}\left(\int_{\Omega} |\nabla \phi |^{p} dx\right)\left(\int_{\Omega}\left(\frac{1}{x_1^p}+\dots+\frac{1}{x_n^p}\right)|\phi|^p dx\right)^{p-1} \ge K\left(\int_{\Omega}\frac{|\phi|^p}{x_1^p+\dots+x_n^p} dx\right)^p,\end{aligned}\]
where $K=\left(\frac{n^{2}(p-1)}{p}\right)^p$.
\end{proof}

\section{Hardy type inequalities with remainders}\label{Section3}

The  purpose of this section is to obtain various remainder terms for the following Hardy type inequality,
\begin{equation}\label{3.1}
\int_{\Omega}|\nabla\phi|^pdx \geq H_p \int_\Omega (\frac{1}{x_1^p}+\frac{1}{x_2^p}+\dots+\frac{1}{x_n^p})|\phi|^p dx,
\end{equation}
where  $H_p$ is a positive constant and will be determined later. Our methods are constructive and give explicit  remainder terms.
\smallskip

Our first result is the following  Hardy-Sobolev-Maz'ya type  inequality.

\begin{theorem}\label{Theorem 3.1}
Let $p>2$, $D=2n$ and  $\phi \in C_0^{\infty}(\Omega)$. Then  we have,
\begin{equation}\label{3.2}
\begin{aligned}
\int_{\Omega}|\nabla\phi|^pdx \geq& \frac{1}{p^p}\int_\Omega \left(\frac{1}{x_1^p}+\frac{1}{x_2^p}+\dots+\frac{1}{x_n^p}\right)|\phi|^p dx\\ &+\frac{1}{2^{p-1}-1}\left(S_p(D)\right)^p\left(\int_\Omega (x_1x_2\dots x_n)^{\frac{-p}{D-p}}|\phi|^{\frac{Dp}{D-p}}dx\right)^{\frac{D-p}{D}}.\end{aligned}
\end{equation}
Moreover, for every $\phi\in C_0^{\infty}(B_R\cap \mathbb{R}_{*}^n)$,

\begin{equation}\label{3.3}
\begin{aligned}
\int_{B_R\cap \mathbb{R}_{*}^n}|\nabla\phi|^pdx \geq \frac{1}{p^p}&\int_{B_R\cap \mathbb{R}_{*}^n} \left(\frac{1}{x_1^p}+\frac{1}{x_2^p}+\dots+\frac{1}{x_n^p}\right)|\phi|^p dx\\&+C\left(\int_{B_R\cap \mathbb{R}_{*}^n} |\phi|^{\frac{Dp}{D-p}}dx\right)^{\frac{D-p}{D}}\end{aligned}
\end{equation}
where $C=\frac{\left(S_p(D)\right)^p}{2^{p-1}-1}(\frac{\sqrt{n}}{R})^{\frac{p}{2}}$.
	
\end{theorem}	
\begin{proof}
Let $\phi \in C_0^{\infty} (\Omega)$ and define  $\phi=(x_1x_2\dots x_n)^{\alpha}\psi$  where $\alpha >0$. Then we have
\[|\nabla \phi|^p=|\alpha (x_1x_2\dots x_n)^{\alpha}(\frac{1}{x_1},\frac{1}{x_2},\dots,\frac{1}{x_n})\psi +(x_1x_2\dots x_n)^{\alpha}\nabla \psi|^p.\]
We now use the following convexity inequality which is valid for any $a, b \in \mathbb{R}^n$ and  $p\ge 2$,
	
\begin{equation} \label{3.4}
|a+b|^p-|a|^p\ge p|a|^{p-2}a\cdot b+c(p)|b|^p
\end{equation} 	
where $c(p)=\frac{1}{2^{p-1}-1}$ (see \cite{Lindqvist}). This yields
\begin{equation}\label{3.5}
\begin{aligned}
|\nabla\phi|^p \ge & \alpha^p (x_1x_2\dots x_n)^{p\alpha}(\frac{1}{x_1^2}+\frac{1}{x_2^2}+\dots+\frac{1}{x_n^2})^{\frac{p}{2}} |\psi|^p \\&+  \alpha^{p-1}(x_1x_2\dots x_n)^{p\alpha}(\frac{1}{x_1^2}+\frac{1}{x_2^2}+\dots+\frac{1}{x_n^2})^{\frac{p-2}{2}} (\frac{1}{x_1},\frac{1}{x_2},\dots,\frac{1}{x_n}) \cdot \nabla(|\psi|^p)  \\&+\frac{1}{2^{p-1}-1}(x_1x_2\dots x_n)^{p\alpha}|\nabla \psi|^p.
\end{aligned}
\end{equation}
Applying integration by parts over $\Omega$ gives
\begin{equation}\label{3.6}
\begin{aligned}
\int_{\Omega}|\nabla\phi|^p dx \ge &  \alpha^p  \int_{\Omega} (x_1x_2\dots x_n)^{p\alpha}(\frac{1}{x_1^2}+\frac{1}{x_2^2}+\dots+\frac{1}{x_n^2})^{\frac{p}{2}}|\psi|^p dx\\&-\alpha^{p-1}\int_{\Omega}\text{div}\left((x_1\dots x_n)^{p\alpha}(\frac{1}{x_1^2}+\dots+\frac{1}{x_n^2})^{\frac{p-2}{2}}(\frac{1}{x_1},\dots,\frac{1}{x_n})\right) |\psi|^p dx \\&+\frac{1}{2^{p-1}-1}\int_{\Omega}(x_1x_2\dots x_n)^{p\alpha}|\nabla\psi|^p dx.
\end{aligned}
\end{equation}
Let us denote the divergence term on the right hand side of \eqref{3.6} by $M$.  A straightforward computation shows that
\begin{equation}\label{3.7}
\begin{aligned}
M:=&-\alpha^{p-1}\text{div}\left((x_1x_2\dots x_n)^{p\alpha}(\frac{1}{x_1^2}+\dots+\frac{1}{x_n^2})^{\frac{p-2}{2}}(\frac{1}{x_1},\dots,\frac{1}{x_n}) \right)  \\ \ge &-\alpha^{p-1}(\alpha p -1)(x_1x_2\dots x_n)^{p\alpha}(\frac{1}{x_1^2}+\frac{1}{x_2^2}+\dots+\frac{1}{x_n^2})^{\frac{p}{2}}.
\end{aligned}
\end{equation}
Substituting (\ref{3.7}) into (\ref{3.6}) gives

\begin{equation}\label{3.8}
\begin{aligned}
\int_{\Omega} |\nabla \phi|^p dx \ge& \left((1-p)\alpha^{p}+\alpha^{p-1}\right) \int_{\Omega} (x_1x_2\dots x_n)^{ p\alpha }(\frac{1}{x_1^2}+\frac{1}{x_2^2}+\dots+\frac{1}{x_n^2})^{\frac{p}{2}}|\psi|^p dx\\&+\frac{1}{2^{p-1}-1}\int_{\Omega}(x_1x_2\dots x_n)^{p\alpha}|\nabla\psi|^p dx.
\end{aligned}\end{equation}
Optimizing the coefficient in front of the first term on the right hand side, one gets  $\alpha=\frac{1}{p}$. Substituting this value of $\alpha$  into \eqref{3.8}, we get

\begin{equation}\label{3.9}
\begin{aligned}
\int_{\Omega}|\nabla \phi|^p dx \ge &\frac{1}{p^p} \int_{\Omega} (x_1x_2\dots x_n) (\frac{1}{x_1^2}+\frac{1}{x_2^2}+\dots+\frac{1}{x_n^2})^{\frac{p}{2}}|\psi|^p dx\\&+ \frac{1}{2^{p-1}-1} \int_{\Omega} (x_1x_2\dots x_n)|\nabla\psi|^p dx.
\end{aligned}
\end{equation}
Furthermore, using the inequality  \[\left(\frac{1}{x_1^2}+\frac{1}{x_2^2}+\dots+\frac{1}{x_n^2}\right)^{\frac{p}{2}}\ge \frac{1}{x_1^p}+\frac{1}{x_2^p}+\dots+\frac{1}{x_n^p} \] we conclude that
	
\begin{equation}\label{3.10}
\begin{aligned}
\int_{\Omega}|\nabla \phi|^p dx \ge &\frac{1}{p^p} \int_{\Omega}  \left(\frac{1}{x_1^p}+\frac{1}{x_2^p}+\dots+\frac{1}{x_n^p}\right)|\phi|^p dx\\&+ \frac{1}{2^{p-1}-1} \int_{\Omega} (x_1x_2\dots x_n)|\nabla\psi|^p dx.
\end{aligned}
\end{equation}	
Applying the  weighted Sobolev inequality (\ref{1.15}) to the
second integral term on the right hand side of \eqref{3.10}, we get the desired inequality \eqref{3.2},

\begin{equation*}
\begin{aligned}
\int_{\Omega}|\nabla \phi|^p dx \ge& \frac{1}{p^p} \int_{\Omega} \left(\frac{1}{x_1^p}+\frac{1}{x_2^p}+\dots+\frac{1}{x_n^p}\right)|\phi|^p dx\\ &+ \frac{1}{2^{p-1}-1}\left(S_p(D)\right)^p\left(\int_\Omega (x_1x_2\dots x_n)^{\frac{-p}{2n-p}}|\phi|^{\frac{Dp}{D-p}}dx\right)^{\frac{D-p}{D}}.
\end{aligned}
\end{equation*}

Furthermore, with the help of arithmetic-geometric mean inequality, we can easily show that
\[(x_1x_2\dots x_n)^{\frac{-p}{2n-p}} \ge (\frac{\sqrt{n}}{r})^{\frac{np}{2n-p}}> (\frac{\sqrt{n}}{R})^{\frac{np}{2n-p}} .\]
Hence, for any $\phi\in C_0^{\infty}(B_R\cap \mathbb{R}_{*}^n)$, we have the desired inequality \eqref{3.3},
\begin{equation*}
\begin{aligned}
\int_{B_R\cap \mathbb{R}_{*}^n}|\nabla\phi|^pdx \geq \frac{1}{p^p}&\int_{B_R\cap \mathbb{R}_{*}^n} \left(\frac{1}{x_1^p}+\frac{1}{x_2^p}+\dots+\frac{1}{x_n^p}\right)|\phi|^p dx\\&+C\left(\int_{B_R\cap \mathbb{R}_{*}^n} |\phi|^{\frac{Dp}{D-p}}dx\right)^{\frac{D-p}{D}}\end{aligned}
\end{equation*}
where $C=\frac{\left(S_p(D)\right)^p}{2^{p-1}-1}(\frac{\sqrt{n}}{R})^{\frac{p}{2}}$.
\end{proof}
\smallskip

Using Theorem \ref{Theorem 2.3} and the method described above, we obtain the following Hardy-Poincar\'e type inequality.
\begin{theorem}\label{Theorem 3.2}
	Let  $\phi \in C_0^{\infty}(B_R \cap \mathbb{R}_*^n)$ and $p>2$. Then we have,
	\begin{equation}\label{3.11}
	\begin{aligned}
	\int_{B_R \cap \mathbb{R}_*^n}|\nabla\phi|^pdx \geq& \frac{1}{p^p}\int_{B_R \cap \mathbb{R}_*^n} \left(\frac{1}{x_1^p}+\frac{1}{x_2^p}+\dots+\frac{1}{x_n^p}\right)|\phi|^p dx\\ &+(\frac{1}{2^{p-1}-1})(\frac{2n}{pR})^p\int_{B_R \cap \mathbb{R}_*^n}|\phi|^p dx. \end{aligned}
	\end{equation}
	
\end{theorem}
\begin{proof}
	Let $\phi \in C_0^{\infty} (B_R \cap \mathbb{R}_*^n)$ and define  $\psi=(x_1x_2\dots x_n)^{-\frac{1}{p}}\phi$.
	Using the same argument as in Theorem \ref{Theorem 3.1}, we have the following inequality (see (\ref{3.10}))
	
	\begin{equation}\label{3.12}
	\begin{aligned}
	\int_{B_R \cap \mathbb{R}_*^n}|\nabla \phi|^p dx \ge &\frac{1}{p^p} \int_{B_R \cap \mathbb{R}_*^n}  \left(\frac{1}{x_1^p}+\frac{1}{x_2^p}+\dots+\frac{1}{x_n^p}\right)|\phi|^p dx\\&+ \frac{1}{2^{p-1}-1} \int_{B_R \cap \mathbb{R}_*^n} (x_1x_2\dots x_n)|\nabla\psi|^p dx.
	\end{aligned}
	\end{equation}	
	An application of the inequality (\ref{2.12}) to the second term on the right-hand side of (\ref{3.12}) which gives the desired inequality 		
	\begin{equation*}
	\begin{aligned}
	\int_{B_R \cap \mathbb{R}_*^n}|\nabla\phi|^pdx \geq& \frac{1}{p^p}\int_{B_R \cap \mathbb{R}_*^n} \left(\frac{1}{x_1^p}+\frac{1}{x_2^p}+\dots+\frac{1}{x_n^p}\right)|\phi|^p dx\\ &+(\frac{1}{2^{p-1}-1})(\frac{2n}{pR})^p\int_{B_R \cap \mathbb{R}_*^n}|\phi|^p dx. \end{aligned}
	\end{equation*}
	
\end{proof} 
	
We now present another theorem that allows us to obtain different remainder terms for the Hardy type inequality \eqref{3.1}.

\begin{theorem}\label{Theorem 3.3}
Let $p>2$. Let $\delta$ be  nonnegative function  and $\delta(x)\in C^{2}(\Omega)$ such
that $-\emph{div}\left((x_1x_2\dots  x_n) \frac{|\nabla \delta|^{p-2}}{\delta^{p-2}} \nabla \delta\right)\ge 0$ in the sense of distributions. Then for all $\phi \in C_0^{\infty}(\Omega)$, we have
\begin{equation}\label{3.13}
\begin{aligned}
\int_{\Omega}|\nabla\phi|^pdx \geq \frac{1}{p^p} \int_{\Omega} \left(\frac{1}{x_1^p}+\frac{1}{x_2^p}+\dots+\frac{1}{x_n^p}\right)|\phi|^p dx+(\frac{1}{2^{p-1}-1})\frac{1}{p^p}\int_{\Omega} \frac{|\nabla\delta|^p}{\delta^p}|\phi|^p dx.
\end{aligned}
\end{equation}

\end{theorem}	
\begin{proof}
Let $\phi \in C_0^{\infty} (\Omega)$ and define  $\psi=(x_1x_2\dots x_n)^{-\frac{1}{p}}\phi$.
Using the same argument as in Theorem $\ref{Theorem 3.1}$, we have the following inequality (see (\ref{3.10}))
\begin{equation}\label{3.14}
\begin{aligned}
\int_{\Omega}|\nabla \phi|^p dx \ge &\frac{1}{p^p} \int_{\Omega} \left(\frac{1}{x_1^p}+\frac{1}{x_2^p}+\dots+\frac{1}{x_n^p}\right)|\phi|^p dx\\&+ \frac{1}{2^{p-1}-1} \int_{\Omega} (x_1x_2\dots x_n)|\nabla\psi|^p dx.
\end{aligned}
\end{equation}
Now we focus on the second term on the right-hand side of this inequality.  Let us define a new variable $\varphi (x):=\delta^{-1/p}\psi(x)$. It follows from the convexity inequality (\ref{3.4}) that
\[|\nabla\psi|^p \ge \frac{|\varphi|^p}{p^p}\delta^{1-p}|\nabla
\delta|^p+p^{1-p}|\nabla\delta|^{p-2} \delta^{2-p}\nabla\delta\cdot \nabla(|\varphi|^p)\] and therefore
\begin{equation}\label{3.15}
\begin{aligned}\int_{\Omega} (x_1x_2 \dots  x_n) |\nabla\psi|^p dx \ge&
\frac{1}{p^p}\int_{\Omega} (x_1x_2\dots x_n) |\nabla \delta|^p\delta^{1-p} |\varphi|^p dx\\&+p^{1-p}\int_{\Omega} (x_1x_2 \dots  x_n)\frac{|\nabla\delta|^{p-2}}{\delta^{p-2}} \nabla
\delta\cdot \nabla (|\varphi|^p) dx.\end{aligned}\end{equation}
Applying integration by parts to the second term on the right hand side of (\ref{3.15})  gives
\[\begin{aligned}\int_{\Omega} (x_1x_2\dots  x_n) |\nabla \psi|^p dx \ge&\frac{1}{p^p}\int_{\Omega} (x_1x_2 \dots  x_n)|\nabla \delta|^p\delta^{1-p} |\varphi|^p dx\\&-p^{1-p}\int_{\Omega} \text{div}\left((x_1x_2 \dots  x_n)\frac{|\nabla \delta |^{p-2}}{\delta^{p-2}}\nabla\delta\right)|\varphi|^p dx.\end{aligned}\]
Since $-\text{div}\left((x_1x_2 \dots  x_n)\frac{|\nabla \delta |^{p-2}}{\delta^{p-2}}\nabla \delta\right)\ge 0$  and  $\varphi=(x_1x_2\dots x_n)^{\frac{-1}{p}}\phi\delta^{-1/p}$ then we get
\begin{equation}\label{3.16} \int_{\Omega}(x_1x_2\dots  x_n)|\nabla\psi|^pdx \ge
\frac{1}{p^p}\int_{\Omega} \frac{|\nabla \delta|^p}{\delta^p}
|\phi|^pdx.
\end{equation}
Substituting (\ref{3.16}) into (\ref{3.15}) gives the desired inequality
\[\int_{\Omega}|\nabla\phi|^pdx \geq \frac{1}{p^p} \int_{\Omega} \left(\frac{1}{x_1^p}+\frac{1}{x_2^p}+\dots+\frac{1}{x_n^p}\right)|\phi|^p dx+(\frac{1}{2^{p-1}-1})\frac{1}{p^p}\int_{\Omega} \frac{|\nabla\delta|^p}{\delta^p}|\phi|^p dx.\]
	
\end{proof}
\smallskip

\subsection{Applications of Theorem \ref{Theorem 3.3}} By applying Theorem \ref{Theorem 3.3} to particular $\delta(x)$, we obtain several Hardy type inequalities with remainder terms. First, let us consider the function,
\[\delta(x):=e^{-(2^{p-1}-1)^{\frac{1}{p}}n|x|}.\]
Note that $\delta$ fulfills the assumptions  of Theorem $\ref{Theorem 3.3}$ when  $|x|\le  \frac{2n-1}{n}\left(\frac{1}{2^{p-1}-1}\right)^{\frac{1}{p}}.$ Hence, we have the  following Hardy type inequality, which has a Poincar\'e remainder term.

\begin{corollary}\label{Corollary 3.1}
Let $2<p<n$ and $R= \frac{2n-1}{n}(\frac{1}{2^{p-1}-1})^{\frac{1}{p}}$. Then for all $\phi \in C_0^{\infty}(B_R \cap \mathbb{R}_*^n)$, we have
\begin{equation}\label{3.17}
\begin{aligned}\int_{B_R \cap \mathbb{R}_*^n}|\nabla \phi|^p dx \ge& \frac{1}{p^p} \int_{B_R \cap \mathbb{R}_*^n} \left(\frac{1}{x_1^p}+\frac{1}{x_2^p}+\dots+\frac{1}{x_n^p}\right)|\phi|^p dx\\&+ \left(\frac{n}{p}\right)^p\int_{B_R \cap \mathbb{R}_*^n} |\phi|^p dx.\end{aligned}
\end{equation}

\end{corollary}
\smallskip

On the other hand, by making the choice
\[\delta(x):= \left(\ln\frac{1}{|x|}\right)^{(2^{p-1}-1)^{\frac{1}{p}}(p-1)},\] gives the following Hardy-Leray type inequality.
\begin{corollary}\label{Corollary 3.2} Let $2<p<n$ and $R= e^{\frac{(1-p)((2^{p-1}-1)^{1/p}-1)}{2n-p}}$. Then for all $\phi \in C_0^{\infty}(B_R \cap \mathbb{R}_*^n)$, we have
\begin{equation}\label{3.18}
\begin{aligned}\int_{B_R \cap \mathbb{R}_*^n}|\nabla \phi|^p dx \ge& \frac{1}{p^p} \int_{B_R \cap \mathbb{R}_*^n}  \left(\frac{1}{x_1^p}+\frac{1}{x_2^p}+\dots+\frac{1}{x_n^p}\right)|\phi|^p dx\\&+ \left(\frac{p-1}{p}\right)^p\int_{B_R \cap \mathbb{R}_*^n} \frac{|\phi|^p}{|x|^p\left(\ln(\frac{1}{|x|})\right)^p}  dx.\end{aligned}
\end{equation}
	
\end{corollary}
\smallskip

Finally, let us consider the function
\[\delta(x):= |x|^{(p-n)(2^{p-1}-1)^{\frac{1}{p}}}.\]
In fact, this choice combines both the Hardy \eqref{1.1} and Hardy type \eqref{3.1}  inequalities into a single inequality.

\begin{corollary}\label{Corollary 3.3} Let $2<p<n$. Then for all $\phi \in C_0^{\infty}(\Omega)$, we have
\begin{equation}\label{3.19}
\begin{aligned}\int_{\Omega}|\nabla \phi|^p dx \ge& \frac{1}{p^p} \int_{\Omega}  \left(\frac{1}{x_1^p}+\frac{1}{x_2^p}+\dots+\frac{1}{x_n^p}\right)|\phi|^p dx\\&+ \left(\frac{n-p}{p}\right)^p\int_{\Omega} \frac{|\phi|^p}{|x|^p}  dx.\end{aligned}
\end{equation}
	
\end{corollary}
\smallskip

\section{Maz'ya type inequalities with remainder terms}\label{Section4}
In this final section we assume that  $\Omega$ is a  bounded domain with smooth boundary $\partial \Omega$ in the upper half space $\mathbb{R}_{+}^n$. The purpose of this section is to obtain various Maz'ya type inequalities with remainder terms. Furthermore, we present some interpolation inequalities which will be useful in what follows. To motivate our discussion, let us start with the following family of Maz'ya-Sobolev inequalities.

\begin{lemma} Let $1<p<n+p-1$, $0\le s\le p$ and $\phi \in C_0^{\infty}(\Omega)$. Then we have,
	\begin{equation}\label{4.1}
	\left(\int_{\mathbb{R}_+^n} x_n^{p-1} |\nabla \phi|^p dx\right)^{\frac{1}{p}} \geq C\left(\int_{\mathbb{R}_+^n} \frac{x_n^{(p-1)(p-s)/p}}{(x_{n-1}^2+x_n^2)^{s/2p}}|\phi|^{p^{*}(s)} dx\right)^{\frac{1}{p^{*}(s)}},
	\end{equation}
	where  $D=n+p-1$,  $p^{*}(s)=p\left(\frac{D-s}{D-p}\right)$ and
	$C=\frac{\left(S_p(D)\right)^{\frac{D(p-s)}{p(D-s)}}}{(2p)^{\frac{s}{p^{*}(s)}}}$.
\end{lemma}
\proof The proof follows  by applying H\"older's inequality, then the Maz'ya inequality \eqref{1.3} and the Sobolev inequality  of Cabr\'e and Ros-Orton \eqref{1.15}  with $A_1=\dots=A_{n-1}=0, A_n=p-1$.  \qed
\smallskip

Inspired by the Maz'ya inequality \eqref{1.3}, we present the following Maz'ya type inequality \eqref{4.2} and its improved version (in the sense that nonnegative terms are added on the right hand side of \eqref{4.2}) over the set $B_R\cap \mathbb{R}_+^n$.  The following is our result.
\begin{theorem}\label{Theorem 4.1}
Let $2\le p <n+p-1$. Then for all $\phi \in C_0^{\infty}(\Omega)$, we have
	\begin{equation}\label{4.2}
	\begin{aligned}
	\int_{\Omega} x_n^{p-1}|\nabla \phi|^pdx \ge \frac{1}{p^p}\int_{\Omega} \frac{x_n^{p-1}}{(x_{n-1}^2+x_n^2)^{p/2}}|\phi|^pdx.
	\end{aligned}
	\end{equation}
	Moreover, for every $\phi \in C_0^{\infty}(B_R\cap \mathbb{R}_+^n)$,
	\begin{equation}\label{4.3}
	\begin{aligned}
	\int_{B_R\cap \mathbb{R}_+^n} x_n^{p-1}|\nabla \phi|^pdx \ge& \frac{1}{p^p}\int_{B_R\cap \mathbb{R}_+^n} \frac{x_n^{p-1}}{(x_{n-1}^2+x_n^2)^{p/2}}|\phi|^pdx\\&+C\left(\int_{B_R\cap \mathbb{R}_+^n}x_n^{p-1}(x_{n-1}^2+x_n^2)^{\frac{D}{2(D-p)}}|\phi|^{\frac{Dp}{D-p}}dx\right)^{\frac{D-p}{D}}.
	\end{aligned}	
	\end{equation}
	where $D=n+p-1$ and  $C= \frac{\left(S_p(D)\right)^p}{(2^{p-1}-1)R}$.
\end{theorem}

\begin{proof}
Let $\phi \in C_0^{\infty} (\Omega)$ and define  $\phi=\rho^{\alpha}\psi$  where $\alpha < 0$ and $\rho=(x_{n-1}^2+x_n^2)$. Then we have
\[|\nabla \phi|^p=|\alpha \rho^{\alpha-1}\psi \nabla \rho +\rho^{\alpha}\nabla \psi|^p.\]
We use the convexity inequality (\ref{3.4}) with $a=\alpha \rho^{\alpha-1}\psi\nabla \rho$ and $b=\rho^{\alpha}\nabla \psi$  to get
\begin{equation}\label{4.4}
\begin{aligned}
|\nabla\phi|^p \ge & |\alpha|^p \rho^{p\alpha-p}|\nabla\rho|^p |\psi|^p \\ &+ \alpha |\alpha|^{p-2}\rho^{p\alpha-p+1}|\nabla\rho|^{p-2}\nabla \rho\cdot \nabla(|\psi|^p)  \\&+\frac{1}{2^{p-1}-1}\rho^{p\alpha}|\nabla \psi|^p.
\end{aligned}
\end{equation}
Multiplying both sides of $(\ref{4.4})$ by $x_n^{p-1}$ and integrating by parts over the set $\Omega$ gives,
\begin{equation}\label{4.5}
\begin{aligned}
\int_{\Omega} x_n^{p-1}|\nabla\phi|^p dx \ge &  |\alpha|^p  \int_{\Omega} x_n^{p-1} \rho^{p\alpha-p}|\nabla \rho|^p |\psi|^p dx\\&-\alpha
|\alpha|^{p-2}\int_{\Omega}\text{div}(x_n^{p-1} \rho^{\alpha p-p+1} |\nabla \rho|^{p-2}\nabla\rho ) |\psi|^p dx \\&+\frac{1}{2^{p-1}-1}\int_{\Omega}
x_n^{p-1}\rho^{p\alpha}|\nabla\psi|^p dx.
\end{aligned}
\end{equation}
A direct computation shows that
\begin{equation}\label{4.6}
\begin{aligned}
\text{div}(x_n^{p-1} \rho^{\alpha p-p+1} |\nabla \rho|^{p-2}\nabla\rho )=2^{p-1}x_n^{p-1}\rho^{\alpha p-\frac{p}{2}}(1+2\alpha p)=0
\end{aligned}
\end{equation}
when $\alpha=-\frac{1}{2p}$.
Substituting $(\ref{4.6})$ into $(\ref{4.5})$ and writing $\psi=(x_{n-1}^2+x_n^2)^{\frac{1}{2p}}\phi$, we get
\begin{equation}\label{4.7}
\begin{aligned}
\int_{\Omega} x_n^{p-1}|\nabla \phi|^p dx \ge &\frac{1}{p^p} \int_{\Omega}  \frac{x_n^{p-1}}{(x_{n-1}^2+x_n^2)^{p/2}}|\phi|^p dx\\&+ \frac{1}{2^{p-1}-1} \int_{\Omega} x_n^{p-1} \rho^{-1/2}|\nabla\psi|^p dx.
\end{aligned}
\end{equation}
By dropping the nonnegative last term in \eqref{4.7}, we obtain the desired inequality

\begin{equation*}
\begin{aligned}
\int_{\Omega} x_n^{p-1}|\nabla \phi|^p dx \ge \frac{1}{p^p} \int_{\Omega}  \frac{x_n^{p-1}}{(x_{n-1}^2+x_n^2)^{p/2}}|\phi|^p dx.
\end{aligned}
\end{equation*}

Now let us look at the second part of Theorem \ref{Theorem 4.1}. The inequality \eqref{4.7} is valid  for $\phi \in C_0^{\infty} (B_R\cap \mathbb{R}_+^n)$. Since $\rho^{-1/2}=(x_{n-1}^2+x_n^2)^{-1/2}\ge (x_1^2+\dots+x_n^2)^{-1/2}= \frac{1}{R}$, one can write
	\begin{equation}\label{4.8}
	\begin{aligned}
	\int_{B_R\cap \mathbb{R}_+^n} x_n^{p-1}|\nabla \phi|^p dx \ge &\frac{1}{p^p} \int_{B_R\cap \mathbb{R}_+^n}  \frac{x_n^{p-1}}{(x_{n-1}^2+x_n^2)^{p/2}}|\phi|^p dx\\&+ \frac{1}{(2^{p-1}-1)R} \int_{B_R\cap \mathbb{R}_+^n} x_n^{p-1} |\nabla\psi|^p dx.
	\end{aligned}
	\end{equation}
	We now apply weighted Sobolev inequality (\ref{1.15}) (see also \cite{Nguyen}, Proposition 1.1)  to the second integral term on the right hand side of \eqref{4.9}  and obtain
	\begin{equation}\label{4.9}
	\int_{B_R\cap \mathbb{R}_+^n} x_n^{p-1} |\nabla\psi|^p dx\ge \left(S_p(D)\right)^{p} \left(\int_{B_R\cap \mathbb{R}_+^n}x_n^{p-1}|\psi|^{\frac{Dp}{D-p}}dx\right)^{\frac{D-p}{D}}.
	\end{equation}
Substituting (\ref{4.9}) into (\ref{4.8}) and taking back substitution $\psi=(x_{n-1}^2+x_n^2)^{\frac{1}{2p}}\phi$, we obtain the desired inequality,
	\begin{equation*}
	\begin{aligned}
	\int_{B_R\cap \mathbb{R}_+^n} x_n^{p-1}|\nabla \phi|^pdx \ge& \frac{1}{p^p}\int_{B_R\cap \mathbb{R}_+^n} \frac{x_n^{p-1}}{(x_{n-1}^2+x_n^2)^{p/2}}|\phi|^pdx\\&+\frac{\left(S_p(D)\right)^p}{(2^{p-1}-1)R}\left(\int_{B_R\cap \mathbb{R}_+^n}x_n^{p-1}(x_{n-1}^2+x_n^2)^{\frac{D}{2(D-p)}}|\phi|^{\frac{Dp}{D-p}}dx\right)^{\frac{D-p}{D}}.
	\end{aligned}	
	\end{equation*}
\end{proof}
\smallskip
Furthermore, by applying the Maz'ya-Sobolev inequality \eqref{4.1} in \eqref{4.8}, we obtain the following  Maz'ya type inequality, which has a weighted  Sobolev remainder term.
\begin{corollary}\label{Corollary}
	Let $2\le p <n+p-1$.  Then for all $\phi \in C_0^{\infty}(B_R\cap \mathbb{R}_+^n)$, we have

	\begin{equation}\label{4.10}
	\begin{aligned}
	\int_{B_R\cap \mathbb{R}_+^n} x_n^{p-1}|\nabla \phi|^pdx \ge& \frac{1}{p^p}\int_{B_R\cap \mathbb{R}_+^n}x_n^{p-1} \frac{|\phi|^p}{(x_{n-1}^2+x_n^2)^{p/2}}dx\\&+\frac{C^p}{(2^p-1)R}\left(\int_{B_R\cap \mathbb{R}_+^n}\frac{x_n^{(p-1)(p-s)/p}}{(x_{n-1}^2+x_n^2)^{\frac{s}{2p}}}|\phi|^{p^{*}(s)}dx\right)^{\frac{p}{p^{*}(s)}},
	\end{aligned}
	\end{equation}
	where  $0\le s\le p$, $D=n+p-1$,  $C=\frac{\left(S_p(D)\right)^{\frac{D(p-s)}{p(D-s)}}}{(2p)^{\frac{s}{p^{*}(s)}}}$ and  $p^{*}(s)=p\left(\frac{D-s}{D-p}\right)$.
\end{corollary}
\smallskip

On the other hand, by H\"older's inequality with the conjugate exponents $p/s$, $p/p-s$, Maz'ya type inequality \eqref{4.2} and the Sobolev inequality \eqref{1.15}, we have the following Maz'ya-Sobolev type  inequality.

\begin{corollary}
	Let $2\le p<n+p-1$ and $\phi \in C_0^{\infty}(\Omega)$. Then we have,
	\begin{equation}\label{4.11}
	\left(\int_{\Omega} x_n^{p-1} |\nabla \phi|^p dx\right)^{\frac{1}{p}} \geq C\left(\int_{\Omega} \frac{x_n^{p-1}}{(x_{n-1}^2+x_n^2)^{\frac{s}{2}}}|\phi|^{p^{*}(s)} dx\right)^{\frac{1}{p^{*}(s)}},
	\end{equation}
	where  $0\le s\le p$, $D=n+p-1$,  $p^{*}(s)=p\left(\frac{D-s}{D-p}\right)$ and
	$C=\frac{\left(S_p(D)\right)^{\frac{D(p-s)}{p(D-s)}}}{p^\frac{s}{p^{*}(s)}}$.
\end{corollary}
\smallskip

\begin{remark}
	Note that one can obtain further Maz'ya type inequalities with remainder terms by using the method in Theorem \ref{4.1} and the  inequality  \eqref{4.11}.
\end{remark}
\smallskip

We now present the second main theorem of this section, which allows us to construct additional Maz'ya type inequalities with remainder terms.

\begin{theorem}\label{Theorem 4.2}
Let  $p\ge 2$ and $\delta \in C^2(\Omega)$ be nonnegative function such that \[-\emph{div}( x_n^{p-1} (x_{n-1}^2+x_n^2)^{-1/2}\frac{|\nabla\delta|^{p-2}}{\delta^{p-2}}\nabla\delta)\geq 0,\] in the sense of distributions.
Then for all  $\phi \in C_0^{\infty}(\Omega)$, we have
\begin{equation}\label{4.12}
\begin{aligned}\int_{\Omega}x_n^{p-1} |\nabla \phi|^p dx \ge& \frac{1}{p^p} \int_{\Omega} \frac{x_n^{p-1}}{(x_{n-1}^2+x_n^2)^{p/2}}|\phi|^p  dx\\&+ \left(\frac{1}{2^{p-1}-1}\right)\frac{1}{p^p} \int_{\Omega} x_n^{p-1} \frac{|\nabla \delta|^p}{\delta ^p} |\phi|^p dx.\end{aligned}
\end{equation}
\end{theorem}	
\begin{proof}
Let $\phi \in C_0^{\infty}(\Omega)$ and define $\psi=(x_{n-1}^2+x_n^2)^{\frac{-1}{2p}}\phi$. Using the same argument as in Theorem \ref{Theorem 4.1}, we have the following inequality
\begin{equation}\label{4.13}
\begin{aligned}
\int_{\Omega} x_n^{p-1}|\nabla \phi|^p dx \ge &\frac{1}{p^p} \int_{\Omega}  \frac{x_n^{p-1}}{(x_{n-1}^2+x_n^2)^{p/2}}|\phi|^p dx\\&+ \frac{1}{2^{p-1}-1} \int_{\Omega} x_n^{p-1} \rho^{-1/2}|\nabla\psi|^p dx.
\end{aligned}
\end{equation}
Now we focus on the second term on the right-hand side of this inequality.  Let us define a new variable $\varphi (x):=\delta(x)^{-1/p}\psi(x)$. By using the convexity inequality (\ref{3.4}), one has
\[|\nabla \psi|^p \ge \frac{|\varphi|^p}{p^p}\delta^{1-p}|\nabla
\delta|^p+p^{1-p}|\nabla\delta|^{p-2} \delta^{2-p}\nabla\delta \cdot \nabla(|\varphi|^p).\]
Hence,
\begin{equation}\label{4.14}
\begin{aligned}\int_\Omega x_n^{p-1} \rho^{-1/2} |\nabla \psi|^p dx \ge&
\frac{1}{p^p}\int_\Omega x_n^{p-1} \rho^{-1/2}  |\nabla \delta|^p\delta^{1-p} |\varphi|^p dx\\&+p^{1-p}\int_\Omega x_n^{p-1} \rho^{-1/2}  \frac{|\nabla\delta|^{p-2}}{\delta^{p-2}} \nabla\delta \cdot \nabla (|\varphi|^p)  dx.\end{aligned}\end{equation}
Applying integration by parts to the second term on the right hand side of (\ref{4.14}), then using the differential inequality $-\text{div}( x_n^{p-1} (x_{n-1}^2+x_n^2)^{-1/2}\frac{|\nabla\delta|^{p-2}}{\delta^{p-2}}\nabla\delta)\geq 0$ and taking back substitution $\varphi=\delta^{-1/p}\psi$, we get
\begin{equation}\label{4.15}
\int_\Omega x_n^{p-1} \rho^{-1/2} |\nabla\psi|^pdx \ge
\frac{1}{p^p}\int_\Omega x_n^{p-1} (x_{n-1}^2+x_n^2)^{-1/2} \frac{|\nabla \delta|^p}{\delta^p}
|\psi|^pdx.\end{equation}
Substituting (\ref{4.15}) into (\ref{4.13}) gives
\begin{equation}\label{4.16}
\begin{aligned}
\int_{\Omega} x_n^{p-1}|\nabla \phi|^p dx \ge & \frac{1}{p^p} \int_{\Omega}  \frac{x_n^{p-1}}{(x_{n-1}^2+x_n^2)^{p/2}}|\phi|^p dx\\&+
\left(\frac{1}{2^{p-1}-1}\right)\frac{1}{p^p}\int_\Omega x_n^{p-1} (x_{n-1}^2+x_n^2)^{-1/2} \frac{|\nabla \delta|^p}{\delta^p}
|\psi|^pdx.
\end{aligned}
\end{equation}
Now, taking back substitution $\psi=(x_{n-1}^2+x_n^2)^{\frac{1}{2p}}\phi$,
we obtain the desired inequality
\begin{equation*}
\begin{aligned}\int_{\Omega}x_n^{p-1} |\nabla \phi|^p dx \ge& \frac{1}{p^p} \int_{\Omega} \frac{x_n^{p-1}}{(x_{n-1}^2+x_n^2)^{p/2}} |\phi|^p dx\\&+ \left(\frac{1}{2^{p-1}-1}\right)\frac{1}{p^p} \int_{\Omega} x_n^{p-1} \frac{|\nabla \delta|^p}{\delta ^p} |\phi|^p dx.\end{aligned}
\end{equation*}
\end{proof}
\smallskip

\subsection{Applications of Theorem \ref{Theorem 4.2}} It is worth noting that one can obtain as many Maz'ya type inequalities with remainder terms as one can find a function $\delta$ satisfying the hypothesis in Theorem \ref{Theorem 4.2}. Before proceeding, we would like to mention the following inequalities were given a second name based on the remainder terms we obtained.  Let us begin by considering the function,
\[\delta:=e^{-(2^{p-1}-1)^{\frac{1}{p}}n|x|}.\]
After some computations, we have the following Maz'ya-Poincar\'{e} type inequality.

\begin{corollary} \label{Corollary 4.3} Let $2 < p < n$, $R= \frac{n+p-3}{n(2^{p-1}-1)^{1/p}}$ and $\phi \in C_0^{\infty}(B_R\cap \mathbb{R}_+^n)$. Then  we have,
\begin{equation}\label{4.17}
\begin{aligned}
\int_{B_R\cap \mathbb{R}_+^n}x_n^{p-1} |\nabla \phi|^p dx \ge& \frac{1}{p^p} \int_{B_R\cap \mathbb{R}_+^n} \frac{x_n^{p-1}}{(x_{n-1}^2+x_n^2)^{p/2}}|\phi|^p  dx\\&+ \left(\frac{n}{p}\right)^p\int_{B_R\cap \mathbb{R}_+^n} x_n^{p-1}|\phi|^p dx.
\end{aligned}
\end{equation}
	
\end{corollary}
\smallskip

Another application of Theorem \ref{Theorem 4.2} with the special function
\[\delta:= \left(\ln\frac{1}{|x|}\right)^{(2^{p-1}-1)^{\frac{1}{p}}(p-1)},\]
gives the following Maz'ya-Leray type inequality.
\begin{corollary}\label{Corollary 4.4}
Let $2 < p < n$, $R = e^{\frac{(1-p)((2^{p-1}-1)^{1/p}-1)}{n-2}}$ and $\phi \in C_0^{\infty}(B_R\cap \mathbb{R}_+^n)$. Then we have,
\begin{equation}\label{4.18}
\begin{aligned}\int_{B_R\cap \mathbb{R}_+^n}x_n^{p-1} |\nabla \phi|^p dx \ge& \frac{1}{p^p} \int_{B_R\cap \mathbb{R}_+^n} \frac{x_n^{p-1}}{(x_{n-1}^2+x_n^2)^{p/2}} |\phi|^p dx\\&+ \left(\frac{p-1}{p}\right)^p\int_{B_R\cap \mathbb{R}_+^n} x_n^{p-1}\frac{|\phi|^p}{|x|^p\left(\ln(\frac{1}{|x|})\right)^p }  dx.\end{aligned}
\end{equation}
	
\end{corollary}
\smallskip

Finally, let us consider the following function
\[\delta:= |x|^{(2^{p-1}-1)^{\frac{1}{p}}(p-n)}.\]
Note that $\delta$ satisfies the hypotheses of Theorem \ref{Theorem 4.2} on $\Omega$ when  $\frac{n}{2}+1 \le p < n$.  Hence we have the following inequality, which includes both
Maz'ya and Hardy type inequalities.
\begin{corollary}\label{Corollary 4.5}
Let $\frac{n}{2}+1 \le p < n$ and  $\phi \in C_0^{\infty}(\Omega)$. Then we have,
\begin{equation}\label{4.19}
\begin{aligned}\int_{\Omega}x_n^{p-1} |\nabla \phi|^p dx \ge& \frac{1}{p^p} \int_{\Omega} \frac{x_n^{p-1}}{(x_{n-1}^2+x_n^2)^{p/2}} |\phi|^p dx\\&+ \left(\frac{n-p}{p}\right)^p\int_{\Omega} x_n^{p-1}\frac{|\phi|^p}{|x|^p} dx.\end{aligned}
\end{equation}
	
\end{corollary}

\bibliographystyle{amsalpha}

\end{document}